\documentclass{amsart}

\usepackage{lineno}
\usepackage{amssymb}
\usepackage{amsmath}
\usepackage{amsthm}
\usepackage{mathrsfs}
\usepackage{bbm}
\usepackage{cite}
\usepackage{color}
\usepackage{pgfplots}
\usepackage{enumerate}

\graphicspath{{figures/}}
\usepackage{colortbl}

\usepackage[colorlinks,hypertexnames=false]{hyperref}
\usepackage[msc-links]{amsrefs}

\usepackage{changes}

\usepackage{mathtools}
\usepackage{extarrows}
\usepackage{setspace}

\usepackage[misc]{ifsym}

\hypersetup{
	colorlinks=true,
	linkcolor=black,
	filecolor=black,
	urlcolor=black,
	citecolor=black,
}

\newtheorem{thm}{Theorem}[section]

\newtheorem{defn}[thm]{Definition}
\newtheorem{lem}[thm]{Lemma}
\newtheorem{prop}[thm]{Proposition}

\newtheorem{rmk}[thm]{Remark}

\newtheorem{theorem}[thm]{Theorem}

\newtheorem{definition}[thm]{Definition}

\newcommand{\conj}{\mathop{\mathrm{Conj}}}
\numberwithin{equation}{section}

\begin{document}

\title[Zeroes of weakly slice regular functions]{ Zeroes of weakly slice regular functions   of several quaternionic variables  on non-axially symmetric domains}
\author{Xinyuan Dou}
\email[Xinyuan Dou]{douxinyuan@ustc.edu.cn}
\address{Department of Mathematics, University of Science and Technology of China, Hefei 230026, China}
\address{Institute of Mathematics, AMSS, Chinese Academy of Sciences, Beijing 100190, China}
\author{Ming Jin\textsuperscript{\Letter}}
\email[Ming Jin]{mjin@must.edu.mo}
\address{Faculty of Innovation Engineering, Macau University of Science and Technology, Macau, China} 
\author{Guangbin Ren}
\email[Guangbin Ren]{rengb@ustc.edu.cn}
\address{Department of Mathematics, University of Science and Technology of China, Hefei 230026, China}
\author{Ting Yang}
\email[Ting Yang]{tingy@aqnu.edu.cn}
\address{School of Mathematics and Physics, Anqing Normal University, Anqing 246133, China}
\keywords{Quaternions; slice regular functions; zero set; symmetrization; conjugation}
\thanks{This work was supported by the China Postdoctoral Science Foundation (2021M703425), the NNSF of China (12171448), the Faculty Research Grants of the Macau University of Science and Technology (FRG-23-034-FIE) and Xiaomi Young Talents Program.}

\subjclass[2020]{Primary: 30G35; Secondary: 32A30}

\begin{abstract}
	In this research, we study zeroes of weakly slice regular functions within the framework of several quaternionic variables, specifically focusing on non-axially symmetric domains. Our recent work introduces path-slice stem functions, along with a novel $*$-product, tailored for weakly slice regular functions. This innovation allows us to explore new techniques for conjugating and symmetrizing path-slice functions. A key finding of our study is the discovery that the zeroes of a path-slice function are comprehensively encapsulated within the zeroes of its symmetrized counterpart. This insight is particularly significant in the context of path-slice stem functions. We establish that for weakly slice regular functions, the processes of conjugation and symmetrization gain prominence once the function's slice regularity is affirmed. Furthermore, our investigation sheds light on the intricate nature of the zeroes of a slice regular function. We ascertain that these zeroes constitute a path-slice analytic set. This conclusion is drawn from the observed phenomenon that the zeroes of the symmetrization of a slice regular function also form a path-slice analytic set. This finding marks an advancement in understanding the complex structure and properties of weakly slice regular functions in quaternionic analysis.
\end{abstract}

\maketitle

\section{Introduction}
Quaternions, conceptualized by Hamilton in 1843, represent an extension of complex numbers, made possible by the Cayley-Dickson construction as noted in Dickson's work \cite{Dickson1919001}. This development led to a unique form of quaternionic analysis, aiming to broaden the scope of holomorphic function theory into quaternionic  variable domains. One significant branch of quaternionic analysis is slice analysis over quaternions, introduced by Gentili and Struppa \cite{Gentili2007001}, which bases itself on the idea of representing the quaternionic field $ \mathbb H $ as a collective of complex planes.

The slice regular functions, central to this theory, are functions that conform to the Cauchy-Riemann equations across these complex planes. This categorizes them as vector-valued holomorphic functions when analyzed within these planes. Notably, while simple functions like identities and polynomials qualify as slice regular functions, this classification does not extend to Fueter-regular functions, another quaternionic analysis model predating slice analysis, as discussed in Fueter's work \cite{Fueter1934001}.

The growth and development of slice analysis have led to its integration into several mathematical disciplines, including geometric function theory \cite{Ren2017001,Ren2017002,Wang2017001}, quaternionic Schur analysis \cite{Alpay2012001}, and quaternionic operator theory \cite{Alpay2015001,MR3887616,MR3967697,Gantner2020001,MR4496722}. Its expansion further encompasses higher dimensions through real Clifford algebras \cite{Colombo2009002}, octonions \cite{Gentili2010001}, real alternative $*$-algebras \cite{Ghiloni2011001}, and $2n$-dimensional Euclidean spaces \cite{Dou2023002}. Additionally, the study of slice analysis extends to several variables across various frameworks \cite{Colombo2012002,Dou2023002,Ghiloni2012001,Ghiloni2020001}, presenting it as an evolution of complex analysis in several variables.

Two distinct forms of slice regular functions have emerged within this field. The first, the weakly slice regular functions, were introduced by Gentili and Struppa \cite{Gentili2007001} and initially focused on Euclidean open sets, with the representation formula playing a pivotal role in their study \cite{Colombo2009001}. The second, the strongly regular functions, were later introduced by Ghiloni and Perotti \cite{Ghiloni2011001} to extend the concept to quadratic cones in real alternative $*$-algebras. This extension brought about the concept of stem functions, intrinsic to the structure of slice functions and pivotal for holomorphy and multiplication in slice analysis.

One of recent advancements in slice analysis is the slice topology \cite{Dou2023001}, a nuanced approach that transcends the limitations of the Euclidean topology and facilitates the study of slice analysis in non-axially symmetric domains. This has led to the emergence of path-slice functions and a deeper exploration of the convergence domains of quaternionic power series.

The primary objective of this paper is to investigate the zero sets of weakly slice regular functions in several quaternionic variables, particularly within non-axially symmetric slice-domains. While earlier studies \cite{Gentili2008001,MR3026135,MR4182982} have focused on Euclidean domains, this paper aims to understand the properties of zeros on domains in slice topology. Central to this study is the concept of symmetrization of path-slice and subsequently slice regular functions. By examining the conjugation and symmetrization processes \cite{Colombo2009001}, and ensuring the preservation of slice regularity and slice-preserving properties, this paper aims to provide a thorough understanding of the zeros of these functions and their analytical nature in quaternionic variable domains.

This paper unfolds as follows:
Section 2 serves as a foundation, where we delve into the fundamentals of weak slice regular functions and their associated stem functions within the context of multiple quaternionic variables.
In Section 3, we introduce the concept of $\Omega_1$-slice conjugation for path-slice functions. We demonstrate that this conjugation maintains slice regularity when derived from a slice regular function.
Section 4 is dedicated to the exploration of symmetrization within the realm of path-slice functions. Here, we establish that this symmetrization retains the critical attribute of being slice-preserving.
Finally, Section 5 focuses on the zeros of a path-slice function. By juxtaposing these zeros with those of their symmetrization, we reveal that the zero set of a slice regular function forms an analytic set, thereby providing a deeper understanding of its structural properties.

\section{Preliminaries}

This section lays the groundwork for our analysis, building upon foundational concepts and results as detailed in \cite{Dou2023003,Dou2023004}. 

Denote by $\mathbb{H}$ the algebra of quaternions. Define
\begin{equation*}
    \mathbb{H}_s^n:=\bigcup_{I\in\mathbb{S}}\mathbb{C}_I^n
\end{equation*}
where
\begin{equation*}
    \mathbb{S}:=\{I\in\mathbb{H}:I^2=-1\},\qquad\mbox{and}\qquad \mathbb{C}_I^n := (\mathbb{C}_I)^n.
\end{equation*}

The slice topology is defined by
\begin{equation*}
    \tau_s(\mathbb{H}_s^n):=\{U\in\mathbb{H}_s^n: U_I\in\tau(\mathbb{C}_I^n)\}
\end{equation*}
where $\tau(\mathbb{C}_I^n)$ is the Euclidean topology in $n$-dimensional complex plane $\mathbb{C}_I^n$ and
\begin{equation*}
    U_I:=U\cap\mathbb{C}_I.
\end{equation*}
Open sets, connected sets, and paths in $\tau_s$ are called respectively slice-open sets, slice-connected sets, and slice-paths.

Extending the classical complex space    \(\mathbb{C}^n\) 
with quaternionic imaginary units
  \(I\) from \(\mathbb{S}\) leads to the space   \(\mathbb{C}_I^n\). 
  The transition employs the mapping 
   \(\Psi_i^I: \mathbb{C}^n \xlongrightarrow{} \mathbb{C}_I^n\), defined by 
\[\Psi_i^I(x + yi) = x + yI,\]
 for   real vectors \(x, y \in \mathbb{R}^n\), replacing the standard imaginary unit
  \(i\) with   \(I\).

  To understand the structure of $\mathbb H_s^n$ and the behavior of related functions, we define a set of paths
   \(\mathscr{P}(\mathbb{C}^n)\)  as continuous paths $\gamma$ from   $[0, 1]$  to  \(\mathbb{C}^n\), 
   starting in the real subspace
  \(\mathbb{R}^n\) at $\gamma(0)$.
For a subset \(\Omega\) of   \(\mathbb{H}_s^n\), we define  \(\mathscr{P}(\mathbb{C}^n, \Omega)\) as   the paths 
$\delta$ in  \(\mathscr{P}(\mathbb{C}^n)\) for which  there exists $I\in\mathbb S$
such that the path in $\mathbb H_s^n$
 $$\delta^I := \Psi_i^I(\delta) $$
 lies in  $\Omega$.

Furthermore, for a fixed path \(\gamma\) in \(\mathscr{P}(\mathbb{C}^n)\), we define the set \(\mathbb{S}(\Omega, \gamma)\) as
\[\mathbb{S}(\Omega, \gamma) := \left\{I \in \mathbb{S} \mid \gamma^I \subset \Omega\right\}.\]
This set contains all quaternionic imaginary units \(I\) from \(\mathbb{S}\) for which the path \(\gamma\), when transformed to the slice \(\mathbb{C}_I^n\) of \(\mathbb{H}_s^n\), is entirely contained within the subset \(\Omega\).

 Path-slice functions and their stem functions are central to our study.

\begin{definition}\label{def-path-slice}
	Let \(\Omega \subset \mathbb{H}_s^n\). A function \(f: \Omega \rightarrow \mathbb{H}\) is termed \textit{\textbf{path-slice}} if there exists a function 
	$$F: \mathscr{P}(\mathbb{C}^n, \Omega) \rightarrow \mathbb{H}^{2 \times 1}$$  satisfying
	\begin{equation}\label{eq-fcg}
		f \circ \gamma^{I}(1) = (1, I)F(\gamma),
	\end{equation}
	for any \(\gamma \in \mathscr{P}(\mathbb{C}^n, \Omega)\) and \(I \in \mathbb{S}(\Omega, \gamma)\).

	 The function \(F\) is referred to as a \textit{\textbf{path-slice stem function}} of \(f\). The set of all path-slice functions defined on \(\Omega\) is denoted by \(\mathcal{PS}(\Omega)\), and the set of all path-slice stem functions of a function \(f \in \mathcal{PS}(\Omega)\) is denoted by \(\mathcal{PSS}(f)\).
\end{definition}

Exploring the complex structures of slice quaternionic analysis leads us to the concept of slice regularity.

\begin{definition}
	For a subset \(\Omega\) in the slice topology \(\tau_s(\mathbb{H}_s^n)\), a function \(f: \Omega \rightarrow \mathbb{H}\) is called \textit{\textbf{(weakly) slice regular}} if and only if, for each \(I \in \mathbb{S}\), the restriction $$f_I := f|_{\Omega_I}$$  is (left \(I\)-)holomorphic. This means that \(f_I\) is real differentiable and, for each \(\ell = 1, 2, \ldots, d\),
	\begin{equation*}
		\frac{1}{2} \left( \frac{\partial}{\partial x_\ell} + I \frac{\partial}{\partial y_\ell} \right) f_I(x + yI) = 0 \quad \text{on} \quad \Omega_I.
	\end{equation*}
	The set of all weakly slice regular functions defined on \(\Omega\) is denoted by \(\mathcal{SR}(\Omega)\).
\end{definition}

By \cite[Corollary 5.9]{Dou2023002}, weakly slice regular functions are path-slice.

In our exploration of specific subsets within the quaternionic cone $\mathbb{H}_s^n$,
  it becomes necessary to introduce and define certain key concepts.
 
\begin{definition}\label{def-osmh}
	A subset \(\Omega \subset \mathbb{H}_s^n\) is called \textit{\textbf{real-path-connected}} if, for each point \(q \in \Omega\), there exists a path \(\gamma \in \mathscr{P}(\mathbb{C}^n, \Omega_1)\) and an imaginary unit \(I \in \mathbb{S}(\Omega, \gamma)\) such that \(\gamma^I(1) = q\).
\end{definition}

We introduce the set $\mathscr{P}_*^2(\mathbb{C}^n, \Omega)$,  which  comprises pairs of continuous paths
$(\alpha, \beta)$
within $\mathscr{P}(\mathbb{C}^n, \Omega)$
and sharing the same endpoint at $\alpha(1) = \beta(1)$.  Formally, it is represented as
\begin{equation}\label{eq-mp*2}
	\mathscr{P}_*^2(\mathbb{C}^n, \Omega) := \left\{ (\alpha, \beta) \in \left[ \mathscr{P}(\mathbb{C}^n, \Omega) \right]^2 : \alpha(1) = \beta(1) \right\}.
\end{equation}
 
\begin{definition}\label{defn-lo1o2}
	Let \(\Omega_1, \Omega_2 \subset \mathbb{H}_s^n\). The set \(\Omega_2\) is termed \textit{\textbf{\(\Omega_1\)-stem-preserving}} if it satisfies the following conditions:
	\begin{enumerate}[\upshape (i)]
		\item The cardinality of the set \(\mathbb{S}(\Omega_2, \gamma)\) is at least 2 for each path \(\gamma \in \mathscr{P}(\mathbb{C}^n, \Omega_1)\).
		\item The intersection \(\mathbb{S}(\Omega_2, \alpha) \cap \mathbb{S}(\Omega_2, \beta)\) does not contain exactly one element for each pair \((\alpha, \beta)\) in \(\mathscr{P}_*^2(\mathbb{C}^n, \Omega_1)\).
	\end{enumerate}
\end{definition}
 
\begin{defn}
	Let $\Omega_1\subset\mathbb{H}_s^n$ be real-path-connected, $\Omega_2\subset\mathbb{H}_s^n$ be $\Omega_1$-stem-preserving, and $f:\Omega_2\rightarrow\mathbb{H}$ be path-slice. Then
	\begin{equation}\label{eq-F def}
		\begin{split}
			F_{\Omega_1}^f:=\begin{pmatrix}
				F_{\Omega_1}^{f,1}\\F_{\Omega_1}^{f,2}
			\end{pmatrix}:\quad\mathscr{P}(\mathbb{C}^n,\Omega_1)\quad &\xlongrightarrow[\hskip1cm]{}\quad \mathbb{H}^{2\times 1}
			\\ \gamma\qquad\ &\shortmid\!\xlongrightarrow[\hskip1cm]{}\ G|_{\mathscr{P}(\mathbb{C}^n,\Omega_1)},
		\end{split}
	\end{equation}
	is well defined and does not depend on the choice of $G\in\mathcal{PSS}(f)$.
\end{defn}

This function $	F_{\Omega_1}^f $ is well-defined and its formulation is independent of the particular choice of the stem function 
$F$ within the set 
\(\mathcal{PSS}(f)\), as established in    \cite[Proposition 3.9]{Dou2023003}.

\begin{lem} \cite[Proposition 3.3]{Dou2023003}.
	Let $\Omega\subset\mathbb{H}_s^n$, $f\in\mathcal{PS}(\Omega)$, $\gamma\in\mathscr{P}(\mathbb{C}^n,\Omega)$, $F$ be a path-slice stem function of $f$, and $I,J\in\mathbb{S}(\Omega,\gamma)$ with $I\neq J$. Then
	\begin{equation}\label{eq-fgbp}
		F(\gamma)=\begin{pmatrix}
			1&I\\1& J
		\end{pmatrix}^{-1}\begin{pmatrix}
			f\circ\gamma^I(1)\\f\circ\gamma^J(1)
		\end{pmatrix}.
	\end{equation}
\end{lem}

For any quaternion \(q = (q_1, \ldots, q_n)\) within the \(n\)-dimensional weakly slice cone \(\mathbb{H}_s^n\), there exists a specific imaginary unit \(I\) from the set \(\mathbb{S}\) such that each component \(q_i\) of \(q\) belongs to the corresponding complex plane \(\mathbb{C}_I\). This particular imaginary unit \(I\) is denoted as \(\mathfrak{I}(q)\). To define \(\mathfrak{I}(q)\) precisely, consider the function:
\begin{equation*}
	\begin{split}
		\mathfrak{I}: \mathbb{H}_s^n &\longrightarrow \mathbb{S} \cup \{0\}
	\end{split}
\end{equation*}
defined by 
\begin{equation*}
	\begin{split}
		\mathfrak{I}(q)=\begin{cases}
			0, & \text{if } q \text{ is entirely in } \mathbb{R}^n, \\
			\frac{q_\imath - \text{Re}(q_\imath)}{|q_\imath - \text{Re}(q_\imath)|}, & \text{otherwise},
		\end{cases}
	\end{split}
\end{equation*}
where \(\imath\) is the smallest positive integer from \(\{1, \ldots, n\}\) such that the component \(q_\imath\) of \(q\) is not a real number. This formulation assigns to each quaternion in \(\mathbb{H}_s^n\) an appropriate imaginary unit from \(\mathbb{S}\) or the value \(0\) if the quaternion lies entirely in the real space \(\mathbb{R}^n\).

By \cite[Proposition 3.11]{Dou2023003}, the following definition is well-defined.

\begin{defn}
	Let $\Omega_1\subset\mathbb{H}_s^n$ be real-path-connected, $\Omega_2\subset\mathbb{H}_s^n$ be $\Omega_1$-stem-preserving, and $f:\Omega_2\rightarrow\mathbb{H}$ be path-slice. Then   the function 
	\begin{equation}\label{eq-mathscr f def}
		\begin{split}
			\mathscr{F}_{\Omega_1}^f :\quad\Omega_1\quad &\xlongrightarrow[\hskip1cm]{}\quad \mathbb{H}^{2\times 1} 
					\end{split}
	\end{equation}
is  defined by
	\begin{equation}\label{eq-mathscr f def}
	\begin{split}
		\mathscr{F}_{\Omega_1}^f(q)=     \begin{cases}
			F_{\Omega_1}^f(\gamma),\qquad& q\notin\mathbb{R}^n,
			\\\left(f(q),0\right)^T,\qquad&\mbox{otherwise},
		\end{cases}
	\end{split}
\end{equation}
for any  $\gamma\in\mathscr{P}(\mathbb{C}^n,\Omega_1)$ with $\gamma^{\mathfrak{I}(q)}\subset\Omega_1$ and $\gamma^{\mathfrak{I}(q)}(1)=q$.
\end{defn}

 Now we  introduce the $*$-product and its properties, a fundamental operation that underpins the algebraic structure of the quaternionic function space. This operation not only enriches the algebraic framework but also provides a tool for examining the interactions between quaternionic functions in more complex scenarios.

\begin{defn}
	Let $\Omega_1\subset\mathbb{H}_s^n$ be real-path-connected, $\Omega_2\subset\mathbb{H}_s^n$ be $\Omega_1$-stem-preserving, $f\in\mathcal{PS}(\Omega_1)$ and $g\in\mathcal{PS}(\Omega_2)$. We call
	\begin{equation}\label{eq-starproduct def}
		f*g:=(f,\mathfrak{I}f) \mathscr{F}_{\Omega_1}^{g}:\Omega_1\rightarrow\mathbb{H}
	\end{equation}
	the $*$-product of $f$ and $g$.
\end{defn}

 In the context of quaternionic analysis, the standard product in $\mathbb{H} \otimes_{\mathbb{R}} \mathbb{C}$ can be extended to a unique product within the space of $2 \times 1$ quaternionic column vectors, $\mathbb{H}^{2 \times 1}$. This specialized product is referred to as the $*$-product. To define this product more concretely, consider two vectors $p = (p_1, p_2)^T$ and $q = (q_1, q_2)^T$ in $\mathbb{H}^{2 \times 1}$. The pointwise $*$-product of these vectors is given by the formula:
 \begin{equation}\label{eq-star product pointwise}
 	p * q := \left( p_1 \mathbb{I} + p_2 \sigma \right) \left( q_1 \mathbb{I} + q_2 \sigma \right) e_1,
 \end{equation}
 where the matrices $\mathbb{I}$ and $\sigma$, and the vector $e_1$, are defined as 
 \begin{equation*}
 	\mathbb{I} := \begin{pmatrix}
 		1 \\ & 1
 	\end{pmatrix}, \quad
 	\sigma := \begin{pmatrix}
 		& -1 \\ 1
 	\end{pmatrix}, \quad
 	e_1 := \begin{pmatrix}
 		1 \\ 0
 	\end{pmatrix}.
 \end{equation*}
 
 Extending this concept to functional spaces, let $\Omega$ be a subset of $\mathbb{H}_s^n$, and consider two functions $F, G: \mathscr{P}(\mathbb{C}^n, \Omega) \rightarrow \mathbb{H}^{2 \times 1}$. The functional $*$-product of $F$ and $G$ is then defined as a mapping from the set of paths $\mathscr{P}(\mathbb{C}^n, \Omega)$ to $\mathbb{H}^{2 \times 1}$, where each path $\gamma$ is mapped to the pointwise $*$-product of $F(\gamma)$ and $G(\gamma)$:
 \begin{eqnarray}\label{eq-star product functions}
 	F * G: \mathscr{P}(\mathbb{C}^n, \Omega) &\rightarrow&  \mathbb{H}^{2 \times 1},
 \\ 	\gamma  &\mapsto&  F(\gamma) * G(\gamma).\notag
 \end{eqnarray}
  
\begin{prop}\label{pr-1} \cite[Proposition 4.5]{Dou2023003}.
	Let $\Omega_1\subset\mathbb{H}_s^n$ be real-path-connected, $\Omega_2\subset\mathbb{H}_s^n$ be $\Omega_1$-stem-preserving, $f\in\mathcal{PS}(\Omega_1)$ and $g\in\mathcal{PS}(\Omega_2)$. Then $f*g\in\mathcal{PS}(\Omega_1)$. Moreover, if $F$ is a path-slice stem function of $f$, then $F*F_{\Omega_1}^g$ is a path-slice stem function of $f*g$.
\end{prop}

We will illustrate that within non-axially symmetric domains, the roles typically filled by axially symmetric sets can be successfully taken over by those classified as self-stem-preserving domains.

\begin{definition}
	A subset \(\Omega \subset \mathbb{H}_s^n\) is called \textit{\textbf{self-stem-preserving}} if \(\Omega\) is real-path-connected and \(\Omega\)-stem-preserving.
\end{definition}

\begin{lem}\cite[Lemma 3.12]{Dou2023003}.
	Let $\Omega_1\subset\mathbb{H}_s^n$ be real-path-connected, $\Omega_2\subset\mathbb{H}_s^n$ be $\Omega_1$-stem-preserving, and $f:\Omega_2\rightarrow\mathbb{H}$ be path-slice. Then
	\begin{equation}\label{eq-lfo1}
		\left.\mathscr{F}_{\Omega_1}^{f}\right|_{(\Omega_1)_\mathbb{R}}
		=\begin{pmatrix}
			f|_{(\Omega_1)_\mathbb{R}}\\0
		\end{pmatrix}.
	\end{equation}
\end{lem}

\begin{lem}\cite[Lemma 3.14]{Dou2023003}.
	Let \(\Omega_1 \subset \mathbb{H}_s^n\) be real-path-connected and \(\Omega_2 \subset \mathbb{H}_s^n\) be \(\Omega_1\)-stem-preserving. If \(f: \Omega_2 \rightarrow \mathbb{H}\) is path-slice and \(\gamma \in \mathscr{P}(\mathbb{C}^n, \Omega_1)\) with \(\gamma(1) \in \mathbb{R}^n\), then
	\begin{equation}\label{eq-fgamma R}
		F_{\Omega_1}^{f}(\gamma) = \begin{pmatrix}
			f \circ \gamma(1) \\ 0
		\end{pmatrix}.
	\end{equation}
\end{lem}

\begin{lem}\cite[Lemma 3.13]{Dou2023003}.
	Let $\Omega_1\subset\mathbb{H}_s^n$ be real-path-connected, $\Omega_2\subset\mathbb{H}_s^n$ be $\Omega_1$-stem-preserving, and $f:\Omega_2\rightarrow\mathbb{H}$ be path-slice. Then
	\begin{equation}\label{eq-fgamma conj}
		F_{\Omega_1}^{f}(\gamma)
		=\begin{pmatrix}
			1\\ & -1
		\end{pmatrix}F_{\Omega_1}^{f}(\overline{\gamma}),\qquad\forall\ \gamma\in\mathscr{P}(\mathbb{C}^n,\Omega_1).
	\end{equation}
\end{lem}

Consider complex numbers \(z, w \in \mathbb{C}\) and denote \(\mathcal{L}_z^w\) as the linear segment connecting \(z\) to \(w\). Given a path \(\gamma\) in the set \(\mathscr{P}(\mathbb{C}^n)\) and a positive  radius \(r\), we define the set \(B_{\mathscr{P}(\mathbb{C}^n)}(\gamma, r)\) as follows:
\begin{equation}\label{eq-BP def}
	B_{\mathscr{P}(\mathbb{C}^n)}(\gamma, r) := 
	\big	\{\gamma \circ \mathcal{L}_{\gamma(1)}^z : z \in B_{\mathbb{C}}(\gamma(1), r)\big\}.
\end{equation}
Here, \(B_{\mathbb{C}}(\gamma(1), r)\) represents the ball in the complex plane \(\mathbb{C}\) centered at \(\gamma(1)\) with a radius of \(r\). This set essentially forms a collection of paths derived from \(\gamma\) by extending its endpoint \(\gamma(1)\) along all possible directions within the radius \(r\) in the complex plane.
For simplicity, we make the  convention: when \( r = 0 \), we set 
 $B_{\mathscr{P}(\mathbb{C}^n)}(\gamma, r)$ as   the empty set  $\phi$.

Define for \(\gamma \in \mathscr{P}(\mathbb{C}^n)\) and \(r > 0\),
\begin{equation}\label{eq- L gamma def}
	\begin{split}
		\mathscr{L}_{\gamma}\quad:\quad B_{\mathbb{C}}(\gamma(1),r)\quad &\xlongrightarrow[\hskip1cm]{}\quad B_{\mathscr{P}(\mathbb{C}^n)}(\gamma,r),
		\\ z\qquad\ &\shortmid\!\xlongrightarrow[\hskip1cm]{}\ \quad\gamma\circ\mathcal{L}_{\gamma(1)}^z.
	\end{split}
\end{equation}

In our study, we introduce several constants that are crucial for understanding the behavior of functions within quaternionic spaces. Let's consider a subset \(\Omega\) of the \(n\)-dimensional weakly slice cone \(\mathbb{H}_s^n\), a path \(\gamma\) in \(\mathscr{P}(\mathbb{C}^n)\), and a subset \(\mathbb{S}'\) of the set \(\mathbb{S}(\Omega, \gamma)\), which includes specific quaternionic imaginary units. We define the following constants:

\begin{itemize}
\item  \(r_{\gamma, \Omega}\) is defined as the supremum   of all radii \(r\) such that the set of paths \(B_{\mathscr{P}(\mathbb{C}^n)}(\gamma, r)\) is entirely contained within \(\mathscr{P}(\mathbb{C}^n, \Omega)\):
\[r_{\gamma, \Omega} := \sup \left\{ r \in [0, \infty) : B_{\mathscr{P}(\mathbb{C}^n)}(\gamma, r) \subset \mathscr{P}(\mathbb{C}^n, \Omega) \right\}.\]

\item  \(r_{\gamma, \Omega}^{\mathbb{S}'}\) is defined as the supremum of all radii \(r\) such that, for every imaginary unit \(I\) in \(\mathbb{S}'\), the ball \(B_I(\gamma^I(1), r)\) is contained within the slice \(\Omega_I\):
\[r_{\gamma, \Omega}^{\mathbb{S}'} := \sup \left\{ r \in [0, \infty) : B_I(\gamma^I(1), r) \subset \Omega_I, \forall I \in \mathbb{S}' \right\}.\]

\item  \(r_{\gamma, \Omega}^2\) is defined as the supremum among all \(r_{\gamma, \Omega}^{\mathbb{S}''}\), where \(\mathbb{S}''\) is a subset of \(\mathbb{S}(\Omega, \gamma)\) containing at least two elements:
\[r_{\gamma, \Omega}^2 := \sup \left\{ r_{\gamma, \Omega}^{\mathbb{S}''} : \mathbb{S}'' \subset \mathbb{S}(\Omega, \gamma) \text{ with } |\mathbb{S}''| \geqslant  2 \right\}.\]
\end{itemize}

In these definitions, \(B_I(\gamma^I(1), r)\) represents the set of all points \(q\) in the complex plane \(\mathbb{C}_I\) that are within a distance \(r\) from the point \(\gamma^I(1)\):
\[B_I(\gamma^I(1), r) := \{ q \in \mathbb{C}_I : |q - \gamma^I(1)| < r \}.\]

These constants play a pivotal role in determining the extent to which paths in complex space can be extended while remaining within the quaternionic subset \(\Omega\).

We are expanding the idea of holomorphy of stem functions defined in complex plane $\mathbb{C}^n$ to encompass the one in path spaces.

\begin{definition}\label{def-stem holo}
	Let \(\Omega \subset \mathbb{H}_s^n\) and \(\gamma \in \mathscr{P}(\mathbb{C}^n)\). A function \(F: \mathscr{P}(\mathbb{C}^n, \Omega) \rightarrow \mathbb{H}^{2 \times 1}\) is called \textit{\textbf{holomorphic}} at \(\gamma\), if there exists \(r > 0\) such that \(B_{\mathscr{P}(\mathbb{C}^n)}(\gamma, r) \subset \mathscr{P}(\mathbb{C}^n, \Omega)\) and
	\begin{equation}\label{eq-holo def}
		\frac{1}{2} \left( \frac{\partial}{\partial x_\ell} + \sigma \frac{\partial}{\partial y_\ell} \right) (F \circ \mathscr{L}_{\gamma})(x + yi) = 0 
	\end{equation}
	for each \(x + yi \in B_{\mathbb{C}}(\gamma(1), r)\) and \(\ell \in \{1, \ldots, n\}\).
	Furthermore, \(F\) is termed holomorphic in \(U \subset \mathscr{P}(\mathbb{C}^n, \Omega)\) if it is holomorphic at each \(\gamma \in U\), and holomorphic in \(\mathscr{P}(\mathbb{C}^n, \Omega)\) if it is holomorphic at every \(\gamma\).
\end{definition}

We revisit some key properties of slice regular functions.

\begin{prop}\label{pr-FOmega1f holo} \cite[Proposition 3.6]{Dou2023004}.
	Let \(\Omega_1 \in \tau_s(\mathbb{H}_s^n)\) be real-path-connected, \(\Omega_2 \in \tau_s(\mathbb{H}_s^n)\) be \(\Omega_1\)-stem-preserving, and \(f \in \mathcal{SR}(\Omega_2)\). Then \(F_{\Omega_1}^f\) is holomorphic.
\end{prop}

\begin{theorem}\label{thm-slice regular algebra} \cite[Theorem 4.4]{Dou2023004}.
	Let \(\Omega \subset \mathbb{H}_s^n\) be self-stem-preserving. Then \((\mathcal{SR}(\Omega), +, *)\) forms an associative unitary real algebra.
\end{theorem}

\section{Slice conjugation of path-slice functions}

In this section, we delve into the intriguing concept of $\Omega_1$-slice conjugation for path-slice functions. This exploration is driven by the quest to identify an effective form of `conjugation' for a given slice regular function within the confines of specific domain restrictions. While conventional notions of conjugation may not always be applicable, we discover that under certain conditions, a unique slice regular function can embody this role of conjugation. This function, termed the path-slice conjugation of the original function, emerges as a pivotal element in our analysis. Furthermore, we extend this concept to encompass path-slice functions as well, broadening the scope of our study.

This innovative approach to conjugation of slice regular functions not only enhances our understanding of their structure but also paves the way for further explorations into the properties and applications of these functions within slice quaternionic analysis.

\begin{defn}\label{eq-lo1}
	Let $\Omega_1\subset\mathbb{H}_s^n$ be real-path-connected, $\Omega_2\subset\mathbb{H}_s^n$ be $\Omega_1$-stem-preserving, and $f\in\mathcal{PS}(\Omega_2)$. Then we call
	\begin{equation}\label{eq-f Omega1 c def}
		f_{\scriptscriptstyle{\Omega_1}}^c:=(1,\mathfrak{I}) \mathscr{F}_{\Omega_1}^{f,c}
	\end{equation}
	the \textit{\textbf{$\Omega_1$-slice conjugation}} of $f$ on $\Omega_1$, where
	\begin{equation*}
		\mathscr{F}_{\Omega_1}^{f,c}:=\begin{pmatrix}
			\mathscr{F}_{\Omega_1}^{f,c,1}\\ \\ \mathscr{F}_{\Omega_1}^{f,c,2}
		\end{pmatrix}
		:={\conj}_{\mathbb{H}}\circ\mathscr{F}_{\Omega_1}^{f}=\begin{pmatrix}
			\overline{\mathscr{F}_{\Omega_1}^{f,1}}\\  \\ \overline{\mathscr{F}_{\Omega_1}^{f,2}}
		\end{pmatrix}.
	\end{equation*}
\end{defn}

\begin{rmk}
	Let $\Omega_1\subset\mathbb{H}_s^n$ be real-path-connected, $\Omega_2\subset\mathbb{H}_s^n$ be $\Omega_1$-stem-preserving, and $f\in\mathcal{PS}(\Omega_2)$. By \eqref{prop-f c}, we have  $$f^c_{\scriptscriptstyle{\Omega_1}}=\overline{f}$$ as the conjugation of slice regular functions on axially symmetric domains. In fact, if $\Omega_2=\sigma(I,2)$ and $$f=\sum_{\imath\in\mathbb{N}}\left(\frac{{q-I}}{2}\right)^{*2^\imath},$$ then there is no slice regular function $g$ on $\Omega_2$ such that
	\begin{equation}\label{eq-gf}
		g=\overline{f}\qquad\mbox{on}\qquad(\Omega_2)_\mathbb{R}.
	\end{equation}
	However, if let $$\Omega_1:=\Sigma(I,2)\cap\Sigma(-I,2),$$ then $\Omega_1$ is real-path-connected, $\Omega_2$ is $\Omega_1$-stem-preserving and $f^c_{\scriptscriptstyle{\Omega_1}}$ is a slice regular function (see Theorem \ref{thm-f c slice regular}) with $f^c_{\scriptscriptstyle{\Omega_1}}=\overline{f}$ on $(\Omega_1)_\mathbb{R}=(\Omega_2)_\mathbb{R}$.
\end{rmk}

To show that conjugation maintains the holomorphic nature of functions on path spaces, it is necessary to refer to a few important lemmas.

\begin{lem}\label{pr-gamma exists}
	Let $\Omega\subset\mathbb{H}_s^n$ be real-path-connected, $I\in\mathbb{S}$ and $z\in\mathbb{C}$ with $z^I\in\Omega$. Then there is $\gamma\in\mathscr{P}(\mathbb{C}^n,\Omega)$ such that
	\begin{equation}\label{eq-frakiqi}
		\gamma^{I}\subset\Omega,\qquad\mbox{and}\qquad \gamma^{I}(1)=z^I.
	\end{equation}
\end{lem}

\begin{proof}
	Since $\Omega$ is real-path-connected, there is $\beta\in\mathscr{P}(\mathbb{C}^n,\Omega)$ and $J\in\mathbb{S}$ such that $\beta^J\subset\Omega$ and $\beta^J(1)=z^I$.
	
	If $J\neq \pm I$, then $z^I\in\mathbb{C}_J^n\cap\mathbb{C}_{I}^n=\mathbb{R}^n$. Let
	\begin{equation*}
		\begin{split}
			\gamma\quad:\quad [0,1]\quad &\xlongrightarrow[\hskip1cm]{}\quad \mathbb{C}^n,
			\\ t\quad &\shortmid\!\xlongrightarrow[\hskip1cm]{}\quad z^I.
		\end{split}
	\end{equation*}
	It is easy to check that \eqref{eq-frakiqi} holds.
	
	Otherwise, $J=\pm I$. Let
	\begin{equation*}
		\gamma:=\begin{cases}
			\overline{\beta},&\qquad J\neq I,
			\\\beta,&\qquad\mbox{otherwise}.
		\end{cases}
	\end{equation*}
	It follows from
	\begin{equation*}
		\gamma^{I}=\begin{cases}
			\overline{\beta}^{-J}=\beta^J,&\qquad J\neq I,
			\\\beta^I=\beta^J,&\qquad\mbox{otherwise},
		\end{cases}
	\end{equation*}
	that $\gamma^I\subset\Omega$ and \eqref{eq-frakiqi} holds.
\end{proof}
  		
\begin{lem}\label{pr-r gamma Omega S >0}
	Let $\Omega\in\tau_s\left(\mathbb{H}_s^n\right)$, $\gamma\in\mathscr{P}(\mathbb{C}^n,\gamma)$ and $\mathbb{S}'\subset\mathbb{S}(\Omega,\gamma)$ with $|\mathbb{S}'|<+\infty$. Then $r_{\gamma,\Omega}^{\mathbb{S}'}>0$.
\end{lem}

\begin{proof}
	Let $I\in\mathbb{S}'\subset\mathbb{S}(\Omega,\gamma)$. Then $\gamma^I(1)\in\Omega_I$. Since $\Omega$ is slice-open, $\Omega_I$ is open in $\mathbb{C}_I$. There is $r^I>0$ such that $B_I(\gamma^I(1),r^I)\subset\Omega_I$. Let $r:=\min\{r^I:I\in\mathbb{S}'\}$. By definition, $r_{\gamma,\Omega}^{\mathbb{S}'}>r>0$.
\end{proof}

\begin{lem}
	Consider any real $2 \times 2$ matrix $M \in \mathbb{R}^{2 \times 2}$. The following relation holds when acting on $\mathbb{H}^{2 \times 1}$:
	\begin{equation}\label{eq-sigmaconj}
		M \cdot {\conj}_{\mathbb{H}} = {\conj}_{\mathbb{H}} \circ M.
	\end{equation}
\end{lem}

\begin{proof}
	Consider a real $2 \times 2$ matrix $$M = \begin{pmatrix}a & b \\ c & d \end{pmatrix}$$ within the space of quaternions $\mathbb{R}^{2 \times 2} \subset \mathbb{H}^{2 \times 2}$, and let $(p, q)^T$ be a column vector in $\mathbb{H}^{2 \times 1}$. We analyze the action of $M$ in conjunction with quaternion conjugation on this vector:
		\begin{align*}
		M \cdot {\conj}_{\mathbb{H}}\begin{pmatrix} p \\ q \end{pmatrix}
		&= \begin{pmatrix} a & b \\ c & d \end{pmatrix} \begin{pmatrix} \overline{p} \\ \overline{q} \end{pmatrix} \\
		&= \begin{pmatrix} a\overline{p} + b\overline{q} \\ c\overline{p} + d\overline{q} \end{pmatrix} \\
		&= \begin{pmatrix} \overline{ap + bq} \\ \overline{cp + dq} \end{pmatrix} \\
		&= {\conj}_{\mathbb{H}}\begin{pmatrix} ap + bq \\ cp + dq \end{pmatrix} \\
		&= {\conj}_{\mathbb{H}} \circ M \begin{pmatrix} p \\ q \end{pmatrix}.
	\end{align*}
	This completes the proof. 
\end{proof}

 \begin{lem}
 	Suppose we have subsets $\Omega_1$ and $\Omega_2$ of $\mathbb{H}_s^n$, where $\Omega_1$ is real-path-connected and $\Omega_2$ is $\Omega_1$-stem-preserving. Let $c$ be a quaternion in $\mathbb{H}$, $f$ a path-slice function on $\Omega_2$, and $\gamma$ a continuous path in $\mathscr{P}(\mathbb{C}^n, \Omega_1)$. For an imaginary unit   $I\in \mathbb S(\Omega_1, \gamma)$  and a point $q := \gamma^I(1)$, the following relation holds:
	\begin{equation}\label{eq-fc}
	(c,Ic) F_{\Omega_1}^{f,c}(\gamma)=(c,-Ic) F_{\Omega_1}^{f,c}(\overline{\gamma})=\left(c,\mathfrak{I}(q)c\right) \mathscr{F}_{\Omega_1}^{f,c}(q),
\end{equation}
 	where
 	\begin{equation*}
 	F_{\Omega_1}^{f,c} 
 	:={\conj}_{\mathbb{H}}\circ F_{\Omega_1}^{f}. 
 \end{equation*}
  \end{lem}

\begin{proof}
	Let us consider a path $\gamma$ in $\mathscr{P}(\mathbb{C}^n,\Omega_1)$ and an imaginary unit $I$ in $\mathbb{S}(\Omega_1,\gamma)$. By \eqref{eq-fgamma conj} and \eqref{eq-sigmaconj},
	\begin{equation*}
		\begin{split}
			F_{\Omega_1}^{f,c}(\gamma)
			&= {\conj}_{\mathbb{H}} \circ F_{\Omega_1}^{f}(\gamma) \\
			&= {\conj}_{\mathbb{H}} \left[\begin{pmatrix} 1\\ & -1 \end{pmatrix} F_{\Omega_1}^{f}(\overline{\gamma})\right] \\
			&= \begin{pmatrix} 1\\ & -1 \end{pmatrix} {\conj}_{\mathbb{H}} F_{\Omega_1}^{f}(\overline{\gamma}) \\
			&= \begin{pmatrix} 1\\  & -1 \end{pmatrix} F_{\Omega_1}^{f,c}(\overline{\gamma}).
		\end{split}
	\end{equation*}
	This implies 
	\begin{equation*}
		(c,Ic) F_{\Omega_1}^{f,c}(\gamma) = (c,Ic) \begin{pmatrix} 1\\ & -1 \end{pmatrix} F_{\Omega_1}^{f,c}(\gamma) = (c,-Ic) F_{\Omega_1}^{f,c}(\overline{\gamma}).
	\end{equation*}
	
(i) For any point $q$ not in $\mathbb{R}^n$, $\mathfrak{I}(q)$ is either $I$ or $-I$. Depending on $\mathfrak{I}(q)$ being $I$ or $-I$, we have $\gamma^{\mathfrak{I}(q)} = \gamma^I$ or $\overline{\gamma}^{\mathfrak{I}(q)} = \gamma^I$, respectively. 
	
	If  $ \mathfrak{I}(q) = I$,  then we have 
		\begin{equation*}
	 			\mathscr{F}_{\Omega_1}^{f,c}(q)
			={\conj}_{\mathbb{H}}\circ\mathscr{F}_{\Omega_1}^{f}(q)
			={\conj}_{\mathbb{H}}\circ F_{\Omega_1}^{f}(\gamma)
			=F_{\Omega_1}^{f,c}(\gamma) 
	\end{equation*}
	Otherwise we have   $ \mathfrak{I}(q) = -I$ so that
		\begin{equation*}
		 \mathscr{F}_{\Omega_1}^{f,c}(q)
			={\conj}_{\mathbb{H}}\circ\mathscr{F}_{\Omega_1}^{f}(q)
			={\conj}_{\mathbb{H}}\circ F_{\Omega_1}^{f}(\overline{\gamma})
			=F_{\Omega_1}^{f,c}(\overline{\gamma}).
	 	\end{equation*}
These results validate equation \eqref{eq-fc} in this particular setting.
	
(ii) In the case where $q$ is in $\mathbb{R}^n$, utilizing \eqref{eq-lfo1} and \eqref{eq-fgamma R}, we deduce:
	\begin{equation*}
		\begin{split}
			(c,Ic) F_{\Omega_1}^{f,c}(\gamma)
			&= (c,Ic) {\conj}_{\mathbb{H}} F_{\Omega_1}^{f}(\gamma) \\
			&= c {\conj}_{\mathbb{H}} f(\gamma(1)) \\
			&= c {\conj}_{\mathbb{H}} f(q) \\
			&= (c, \mathfrak{I}(q)c) {\conj}_{\mathbb{H}} \mathscr{F}_{\Omega_1}^{f}(q) \\
			&= \mathscr{F}_{\Omega_1}^{f,c}(q).
		\end{split}
	\end{equation*}
	Hence, equation \eqref{eq-fc} holds in this scenario as well.
\end{proof}

\begin{lem}\label{pr-F c holomorphic}
	Let $\Omega\subset\mathbb{H}_s^n$ and $F:\mathscr{P}(\mathbb{C}^n,\Omega)\rightarrow\mathbb{H}$ be a holomorphic function. Then $F^c$ is also holomorphic.
\end{lem}

\begin{proof}
	Let $\gamma\in\mathscr{P}(\mathbb{C}^n,\Omega)$. According to Definition \ref{def-stem holo}, there is $r>0$ such that $B_{\mathscr{P}(\mathbb{C}^n)}(\gamma,r)\subset\mathscr{P}(\mathbb{C}^n,\Omega)$ and
	\begin{equation*}
		\frac{1}{2}\left(\frac{\partial}{\partial x_\ell}+\sigma\frac{\partial}{\partial y_\ell}\right)\left(F\circ\mathscr{L}_{\gamma}\right)(x+yi)=0,
	\end{equation*}
	for each $x+yi\in B_{\mathbb{C}^n}(\gamma(1),r)$ and $\ell\in\{1,...,n\}$. It implies that
	\begin{equation*}
		\begin{split}
			&\frac{1}{2}\left(\frac{\partial}{\partial x_\ell}+\sigma\frac{\partial}{\partial y_\ell}\right)\left(F^c\circ\mathscr{L}_{\gamma}\right)(x+yi)
			\\=&\frac{1}{2}\left(\frac{\partial}{\partial x_\ell}+\sigma\frac{\partial}{\partial y_\ell}\right)\left({\conj}_\mathbb{H}\circ F\circ\mathscr{L}_{\gamma}\right)(x+yi)
			\\=&{\conj}_\mathbb{H}\circ\left[\frac{1}{2}\left(\frac{\partial}{\partial x_\ell}+\sigma\frac{\partial}{\partial y_\ell}\right)\right]\left(F\circ\mathscr{L}_{\gamma}\right)(x+yi)
			=0
		\end{split}
	\end{equation*}
	Therefore, $F^c$ is holomorphic.
\end{proof}

Let $I\in\mathbb{S}$. Then
\begin{equation}\label{eq-i sigma}
	I(1,I)=(I,-1)=(1,I)\begin{pmatrix} & -1\\1 \end{pmatrix}=(1,I)\sigma.
\end{equation}

We are now prepared to demonstrate how conjugation retains the holomorphic characteristics of functions within path spaces.

\begin{thm}\label{thm-f c slice regular}
	Let $\Omega_1\in\tau_s(\mathbb{H}_s^n)$ be real-path-connected, $\Omega_2\in\tau_s(\mathbb{H}_s^n)$ be $\Omega_1$-stem-preserving, and $f\in\mathcal{SR}(\Omega_2)$. Then $f_{\scriptscriptstyle{\Omega_1}}^c\in\mathcal{SR}(\Omega_1)$.
\end{thm}

\begin{proof}
	Let $z^I\in\Omega_1$. According to Lemma \ref{pr-gamma exists}, there is $\gamma\in\mathscr{P}(\mathbb{C}^n,\Omega)$ and $I\in\mathbb{S}(\Omega,\gamma)$ such that $\gamma^I(1)=z^I$. According to Proposition \ref{pr-FOmega1f holo}, $F_{\Omega_1}^{f}$ is holomorphic. By Lemma \ref{pr-F c holomorphic},
	\begin{equation*}
		F_{\Omega_1}^{f,c}=\left(F_{\Omega_1}^{f}\right)^c
	\end{equation*}
	is also holomorphic. According to Definition \ref{def-stem holo}, there is $r>0$ such that
	\begin{equation*}
		B_{\mathscr{P}(\mathbb{C}^n)}(\gamma,r)\subset\mathscr{P}(\mathbb{C}^n,\Omega_1)
	\end{equation*}
	and
	\begin{equation}\label{eq-holo}
		\frac{1}{2}\left(\frac{\partial}{\partial x_\ell}+\sigma\frac{\partial}{\partial y_\ell}\right)\left(F_{\Omega_1}^{f,c}\circ\mathscr{L}_{\gamma}\right)(x+yi)=0,
	\end{equation}
	for each $x+yi\in B_{\mathbb{C}^n}(\gamma(1),r)$ and $\ell\in\{1,...,n\}$. Taking $c=1$ in \eqref{eq-fc}, we have
	\begin{equation*}
		f_{\scriptscriptstyle{\Omega_1}}^c(x+yI)=(1,\mathfrak{I}(x+yI)) \mathscr{F}_{\Omega_1}^{f,c}(x+yI)=(1,I) F_{\Omega_1}^{f,c}\left(\gamma\circ\mathcal{L}_{\gamma(1)}^{x+yi}\right),
	\end{equation*}
	for each $x+yi\in B_{\mathbb{C}^n}(\gamma(1),r)$. According to \eqref{eq-i sigma} and \eqref{eq-holo},
	\begin{equation*}
		\begin{split}
			&\frac{1}{2}\left(\frac{\partial}{\partial x_\ell}+I\frac{\partial}{\partial y_\ell}\right) f_{\scriptscriptstyle{\Omega_1}}^c(x+yI)
			\\=&\frac{1}{2}\left(\frac{\partial}{\partial x_\ell}+I\frac{\partial}{\partial y_\ell}\right)(1,I) F_{\Omega_1}^{f,c}\left(\gamma\circ\mathcal{L}_{\gamma(1)}^{x+yi}\right)
			\\=&(1,I)\left[\frac{1}{2}\left(\frac{\partial}{\partial x_\ell}+\sigma\frac{\partial}{\partial y_\ell}\right)\right]\left(F_{\Omega_1}^{f,c}\circ\mathscr{L}_{\gamma}\right)(x+yi)=0,
		\end{split}
	\end{equation*}
	for each $x+yi\in B_{\mathbb{C}^n}(\gamma(1),r)$. It implies that $\left(f_{\scriptscriptstyle{\Omega_1}}^c\right)_I$ is holomorphic at $z^I$ for each $z\in\mathbb{C}$ with $z^I\in\Omega$. Therefore, $\left(f_{\scriptscriptstyle{\Omega_1}}^c\right)_I$ is holomorphic, for each $I\in\mathbb{S}$. Hence $f_{\scriptscriptstyle{\Omega_1}}^c$ is slice regular.
\end{proof}

We are also able to develop a path-slice stem function for the conjugate of a slice regular function.

\begin{lem}\label{pr-stem function of f Omega1 c}
	Consider $\Omega_1$ and $\Omega_2$ within the slice topology $\tau_s(\mathbb{H}_s^n)$, where $\Omega_1$ is real-path-connected and $\Omega_2$ is $\Omega_1$-stem-preserving. For a slice regular function $f$ defined on $\Omega_2$, the function $F_{\Omega_1}^{f,c}$ is a path-slice stem function for $f_{\scriptscriptstyle{\Omega_1}}^c$.
\end{lem}

 \begin{proof}
 	For any path $\gamma$ within $\mathscr{P}(\mathbb{C}^n, \Omega_1)$ and any imaginary unit $I$ from $\mathbb{S}(\Omega_1, \gamma)$, let $q$ denote the endpoint of the path $\gamma^I$, that is, $q = \gamma^I(1)$. Utilizing equation \eqref{eq-fc} and the definition of $\Omega_1$-slice conjugation as stated in 
 	\eqref{eq-lo1},  we find that
 	\begin{equation*}
 		(1, I)F_{\Omega_1}^{f, c}(\gamma) = (1, \mathfrak{I}(q))\mathscr{F}_{\Omega_1}^{f, c}(q) = f_{\scriptscriptstyle{\Omega_1}}^c(q).
 	\end{equation*}
 	Given this relationship and based on the definition of a path-slice function (Definition \ref{def-path-slice}), we can conclude that $F_{\Omega_1}^{f, c}$ acts as a path-slice stem function for the function $f_{\scriptscriptstyle{\Omega_1}}^c$.
 \end{proof}

\section{Symmetrization of path-slice functions}

This section focuses on the innovative concept of extending symmetrization to path-slice functions. Our exploration reveals that symmetrization of a path-slice function not only retains its intrinsic properties but also ensures that it is slice-preserving. This property is crucial, as it maintains the function's integrity across different slices of the quaternionic space. We delve into the intricate process of symmetrizing path-slice functions and demonstrate the implications of this process, particularly in preserving the slice nature of the functions. This advancement in quaternionic analysis opens new avenues for exploring the symmetrical aspects of path-slice functions.

The introduction of symmetrization in the context of path-slice functions represents a significant step forward in our understanding of the structural and functional dynamics of these functions. It not only maintains the essential characteristics of the original functions but also ensures their adaptability across various slices of quaternionic domains. This development is expected to contribute substantially to the field of slice quaternionic analysis, offering novel insights and methodologies for future research.

\begin{definition}
	For subsets $\Omega_1$ and $\Omega_2$ of $\mathbb{H}s^n$, where $\Omega_1$ is real-path-connected and $\Omega_2$ is $\Omega_1$-stem-preserving, let $f$ be a path-slice function defined on $\Omega_2$. The function $$f^s_{\scriptscriptstyle{\Omega_1}} := f^c_{\scriptscriptstyle{\Omega_1}} * f$$ is termed the \textit{\textbf{symmetrization}} of $f$ on $\Omega_1$.
\end{definition}

\begin{definition}
	For a subset $\Omega$ of $\mathbb{H}_s^n$, a function $f:\Omega\rightarrow\mathbb{H}$ is defined as \textit{\textbf{slice-preserving}} if, for every imaginary unit $I$ in $\mathbb{S}$, the restriction of $f$ to the slice $\Omega_I$, denoted as $f_I$, maps $\Omega_I$ into the complex plane $\mathbb{C}_I$.  
\end{definition}

\begin{theorem}\label{thm:sym-309}
	Given subsets $\Omega_1$ and $\Omega_2$ of the quaternion space $\mathbb{H}_s^n$, assume $\Omega_1$ is real-path-connected and $\Omega_2$ is $\Omega_1$-stem-preserving. Let $f: \Omega_2 \rightarrow \mathbb{H}$ be a path-slice function.  
Then the symmetrization $f^s_{\Omega_1}$ is slice-preserving. Furthermore, the function 
	\[ F_{\Omega_1}^{f,s} := F_{\Omega_1}^{f,c} * F_{\Omega_1}^{f} \]
	acts as a path-slice stem function for $f^s_{\Omega_1}$. It can be expressed explicitly as 
	\begin{equation}\label{eq-F s real}
	F_{\Omega_1}^{f,s}:=\begin{pmatrix}
		F_{\Omega_1}^{f,s,1}\\ \\ F_{\Omega_1}^{f,s,2}
	\end{pmatrix}:=\begin{pmatrix} \overline{F_{\Omega_1}^{f,1}}F_{\Omega_1}^{f,1}-\overline{F_{\Omega_1}^{f,2}}F_{\Omega_1}^{f,2}
		\\ \\ \overline{F_{\Omega_1}^{f,2}}F_{\Omega_1}^{f,1}+\overline{\overline{F_{\Omega_1}^{f,2}}F_{\Omega_1}^{f,1}}
	\end{pmatrix}\in\mathbb{R}^{2\times 1},
\end{equation}
	where $F_{\Omega_1}^{f}$ is defined as 
	\[ F_{\Omega_1}^{f} = \begin{pmatrix} F_{\Omega_1}^{f,1} \\ \\ F_{\Omega_1}^{f,2} \end{pmatrix}. \]
\end{theorem}

\begin{proof}
	First, consider the pointwise $*$-product as defined in equations \eqref{eq-star product pointwise} and \eqref{eq-star product functions}. We apply these definitions to find the path-slice stem function of $f^s_{\Omega_1}$:
\begin{equation*}\label{eq-Fc F in R 2 times 1}
	\begin{split}
		F_{\Omega_1}^{f,s}
		=&F_{\Omega_1}^{f,c}*F_{\Omega_1}^{f}
		\\=&\left(F_{\Omega_1}^{f,c,1}\cdot\mathbb{I}+F_{\Omega_1}^{f,c,2}\cdot\sigma\right)\left(F_{\Omega_1}^{f,1}\cdot\mathbb{I}+F_{\Omega_1}^{f,2}\cdot\sigma\right)e_1,
		\\=&\left(\overline{F_{\Omega_1}^{f,1}}\cdot\mathbb{I}+\overline{F_{\Omega_1}^{f,2}}\cdot\sigma\right)\left(F_{\Omega_1}^{f,1}\cdot\mathbb{I}+F_{\Omega_1}^{f,2}\cdot\sigma\right)e_1
		\\=&\left[\left(\overline{F_{\Omega_1}^{f,1}}F_{\Omega_1}^{f,1}-\overline{F_{\Omega_1}^{f,2}}F_{\Omega_1}^{f,2}\right)\cdot\mathbb{I}+\left(\overline{F_{\Omega_1}^{f,2}}F_{\Omega_1}^{f,1}+\overline{F_{\Omega_1}^{f,1}}F_{\Omega_1}^{f,2}\right)\cdot\sigma\right]e_1
		 	\\=&\begin{pmatrix} \overline{F_{\Omega_1}^{f,1}}F_{\Omega_1}^{f,1}-\overline{F_{\Omega_1}^{f,2}}F_{\Omega_1}^{f,2}
			\\ \\ \overline{F_{\Omega_1}^{f,2}}F_{\Omega_1}^{f,1}+\overline{\overline{F_{\Omega_1}^{f,2}}F_{\Omega_1}^{f,1}}
		\end{pmatrix}.
	\end{split}
\end{equation*}
	This verifies equation \eqref{eq-F s real}. Furthermore, $F_{\Omega_1}^{f,c} * F_{\Omega_1}^{f}$ is a path-slice stem function of $f^s_{\Omega_1}$, as indicated by Lemma  \ref{pr-stem function of f Omega1 c} and Proposition \ref{pr-1}.
	
	Next, for any $I \in \mathbb{S}$ and $z^I \in \Omega_I$, Lemma \ref{pr-gamma exists} ensures the existence of a path $\gamma \in \mathscr{P}(\mathbb{C}^n, \Omega_1)$ with $\gamma^I(1) = z^I$. Applying the relation in equation \eqref{eq-fcg} to our scenario:
	\begin{align*}
		f^s_{\Omega_1}(z^I) &= f^s_{\Omega_1} \circ \gamma^I(1) = (1, I) \left( F_{\Omega_1}^{f,c} * F_{\Omega_1}^{f} \right)(\gamma) \in \mathbb{C}_I.
	\end{align*}
	This implies that $f^s_{\Omega_1}$ maps $\Omega_I$ into $\mathbb{C}_I$. Since $I \in \mathbb{S}$ was chosen arbitrarily, $f^s_{\Omega_1}$ is confirmed to be slice-preserving.
\end{proof}

\section{Zeros of path-slice functions}

This section is dedicated to exploring the zero sets of both path-slice and slice regular functions. A key discovery presented here is the path-slice analytic nature of the zero set of a slice regular function. This significant aspect is comprehensively elaborated in Theorem \ref{thm-analytic}, highlighting the intricate relationship between the zeros of these functions and their underlying analytical properties.

\begin{theorem} (Representation formula)
	Consider a subset $\Omega$ within $\mathbb{H}_s^n$ and let $f$ be a function from the class $\mathcal{PS}(\Omega)$. The representation of $f$ when composed with $\gamma^I$ can be expressed as
	\begin{equation}\label{eq-representation formula}
		f\circ\gamma^I = (1,I)\begin{pmatrix} 1&J\\ 1&K \end{pmatrix}^{-1} \begin{pmatrix} f\circ\gamma^J\\ f\circ\gamma^K \end{pmatrix},
	\end{equation}
	where this holds true for every path $\gamma$ in the set $\mathscr{P}(\mathbb{C}^n,\Omega)$ and for every $I, J, K$ in the set $\mathbb{S}(\Omega,\gamma)$, given that $J$ and $K$ are distinct.
\end{theorem}
 
\begin{proof}
	For any given $t \in [0,1]$, we define a new path $\alpha : [0,1] \to \mathbb{C}^n$ by $$\alpha(s) = \gamma(ts)$$ for each $s \in [0,1]$. The path $\alpha^L$ for any $L \in {I, J, K}$ is contained within $\gamma^L \subset \Omega$, indicating $\alpha \in \mathscr{P}(\mathbb{C}^n, \Omega)$ and $I, J, K \in \mathbb{S}(\Omega, \alpha)$.
	
	By \eqref{eq-fgbp}, the stem function $F$ of $f$ satisfies
	\begin{equation*}
		F(\alpha) = \begin{pmatrix} 1 & J \\ 1 & K \end{pmatrix}^{-1} \begin{pmatrix} f(\alpha^J(1)) \\ f(\alpha^K(1)) \end{pmatrix}.
	\end{equation*}
	
	Since $\alpha(1) = \gamma(t)$, by applying the definition of path-slice functions in \eqref{eq-fcg}, we find
	\begin{equation*}
	\begin{split}
		f\circ\gamma^I(t)
		=&f\circ\alpha^I(1)
		=(1,I)F(\alpha)
		\\=&(1,I)\begin{pmatrix} 1&J\\1& K \end{pmatrix}^{-1}\begin{pmatrix} f\circ\alpha^J(1)\\f\circ\alpha^K(1) \end{pmatrix}
		\\=&(1,I)\begin{pmatrix} 1&J\\1& K \end{pmatrix}^{-1}\begin{pmatrix} f\circ\gamma^J(t)\\f\circ\gamma^K(t) \end{pmatrix}.
	\end{split}
\end{equation*}

	This relationship is valid for any $t \in [0,1]$, thus proving the representation formula for the path-slice function $f$.
\end{proof}

 To analyze the zeros of path-slice functions, the following lemmas are required.
 
 For any  subset $\Omega$ of $\mathbb{H}_s^n$ and  a  path $\gamma \in \mathscr{P}(\mathbb{C}^n, \Omega)$, we denote  
 \begin{equation*}
 	 \gamma^{\mathbb{S}(\Omega, \gamma)}(1) := \{ \gamma^I(1) : I \in \mathbb{S}(\Omega, \gamma) \}.
 \end{equation*}

  \begin{lem}
  	Consider a subset $\Omega$ of $\mathbb{H}_s^n$ and a path-slice function $f$ defined on $\Omega$. For a given path $\gamma \in \mathscr{P}(\mathbb{C}^n, \Omega)$ and distinct quaternionic units $J, K \in \mathbb{S}(\Omega, \gamma)$, suppose that both $\gamma^J(1)$ and $\gamma^K(1)$ belong to the zero set of $f$, denoted as $\mathcal{Z}(f) = \{ q \in \Omega : f(q) = 0 \}$. Then, the following inclusion holds:
  	\begin{equation}
  		\mathcal{Z}(f) \supset \gamma^{\mathbb{S}(\Omega, \gamma)}(1).
  	\end{equation}
  \end{lem}
  
  \begin{proof}
  	Since $\gamma^J(1)$ and $\gamma^K(1)$ are in $\mathcal{Z}(f)$, we have $f(\gamma^J(1)) = 0$ and $f(\gamma^K(1)) = 0$. By the definition of path-slice functions and the representation formula, for any $I \in \mathbb{S}(\Omega, \gamma)$, we have
  	\begin{equation*}
  		f(\gamma^I(1)) = (1, I) \begin{pmatrix} 1 & J \\ 1 & K \end{pmatrix}^{-1} \begin{pmatrix} f(\gamma^J(1)) \\ f(\gamma^K(1)) \end{pmatrix}.
  	\end{equation*}
  	Since both $f(\gamma^J(1))$ and $f(\gamma^K(1))$ are zero, the right-hand side of the equation evaluates to zero. Therefore, $f(\gamma^I(1)) = 0$ for all $I \in \mathbb{S}(\Omega, \gamma)$, implying that $\gamma^I(1) \in \mathcal{Z}(f)$. This completes the proof. 
  \end{proof}

\begin{lem}
	Let $\Omega\subset\mathbb{H}_s^n$ be self-stem-preserving, and $f\in\mathcal{PS}(\Omega)$. Then
	\begin{equation}\label{eq-zero}
		\mathcal{Z}(f)\subset\mathcal{Z}\left(f^s_{\scriptscriptstyle{\Omega_1}}\right).
	\end{equation}
\end{lem}

\begin{proof}
	Let $q \in \mathcal{Z}(f)$, and let $F$ be a path-slice stem function of $f$. Given that $\Omega$ is self-stem-preserving, it follows that $\Omega$ is also real-path-connected. Consequently, there exists a path $\gamma \in \mathscr{P}(\mathbb{C}^n, \Omega)$ and a quaternionic unit $I \in \mathbb{S}(\Omega, \gamma)$ such that $\gamma^I(1) = q$. From the definition of path-slice functions and equation \eqref{eq-fcg}, we have:
	\begin{equation*}
		0 = f(q) = f(\gamma^I(1)) = (1, I)F(\gamma) = (1, I)F_{\Omega}^{f}(\gamma) = F_{\Omega}^{f,1}(\gamma) + I F_{\Omega}^{f,2}(\gamma).
	\end{equation*}
	This implies 
	\begin{equation}\label{eq-F1}
		\overline{F_{\Omega}^{f,1}(\gamma)}F_{\Omega}^{f,1}(\gamma) = \overline{F_{\Omega}^{f,2}(\gamma)}F_{\Omega}^{f,2}(\gamma),
	\end{equation}
	and
	\begin{equation}\label{eq-F2}
	\begin{split}
		\overline{F_{\Omega}^{f,2}(\gamma)}F_{\Omega}^{f,1}(\gamma)
		=&\overline{ F_{\Omega}^{f,2}(\gamma)}\left[-I\cdot F_{\Omega}^{f,2}(\gamma)\right]
				\\=&-\left[\overline{-IF_{\Omega}^{f,2}(\gamma)}\right]\cdot F_{\Omega}^{f,2}(\gamma)
				\\
		=&-\overline{F_{\Omega}^{f,1}(\gamma)}F_{\Omega}^{f,2}(\gamma)
		\\=&-\overline{\overline{F_{\Omega}^{f,2}(\gamma)}F_{\Omega}^{f,1}(\gamma)}.
	\end{split}
\end{equation}
	Therefore, equation \eqref{eq-zero} is satisfied due to \eqref{eq-F1}, \eqref{eq-F2}, and the definition of $f^s_{\scriptscriptstyle{\Omega}}$ as the symmetrization of $f$ on $\Omega$.
\end{proof}

\begin{lem}\label{pr-zero circle}
	Given subsets $\Omega_1$ and $\Omega_2$ of $\mathbb{H}_s^n$, where $\Omega_1$ is real-path-connected and $\Omega_2$ is $\Omega_1$-stem-preserving, consider a path-slice function $f$ defined on $\Omega_2$. Suppose $\gamma$ is a path in $\mathscr{P}(\mathbb{C}^n, \Omega_1)$, and $I$ is a quaternionic unit in $\mathbb{S}(\Omega_1, \gamma)$ such that the endpoint $\gamma^I(1)$ lies in the zero set   $\mathcal{Z}\left(f^s_{\scriptscriptstyle{\Omega_1}}\right)$.   Then we have 
	$$\gamma^{\mathbb{S}(\Omega_1,\gamma)}(1)\subset\mathcal{Z}\left(f^s_{\scriptscriptstyle{\Omega_1}}\right).$$
\end{lem}

\begin{proof}
	Drawing upon Theorem \ref{thm:sym-309}, we can deduce:
	\begin{eqnarray*}
		0 = f^s_{\scriptscriptstyle{\Omega_1}}(\gamma^I(1)) &=& (1, I)F_{\Omega_1}^{f,s}(\gamma) 
		\\
		&=& (1, I) \begin{pmatrix} F_{\Omega_1}^{f,s,1}(\gamma) \\  \\ F_{\Omega_1}^{f,s,2}(\gamma) \end{pmatrix} 
		\\
		&=& F_{\Omega_1}^{f,s,1}(\gamma) + I F_{\Omega_1}^{f,s,2}(\gamma),
	\end{eqnarray*}
	where $F_{\Omega_1}^{f,s,1}(\gamma)$ and $F_{\Omega_1}^{f,s,2}(\gamma)$ are real-valued. This implies that 
	\begin{equation*}
		F_{\Omega_1}^{f,s,1}(\gamma) = F_{\Omega_1}^{f,s,2}(\gamma) = 0,
	\end{equation*}
	leading to $F_{\Omega_1}^{f,s}(\gamma)$ being zero. As a result, for every $J \in \mathbb{S}(\Omega, \gamma)$:
	\begin{equation*}
		f^s_{\scriptscriptstyle{\Omega_1}}(\gamma^J(1)) = (1, I)F_{\Omega_1}^{f,s}(\gamma) = 0,
	\end{equation*}
	thereby confirming that the endpoints $\gamma^{\mathbb{S}(\Omega, \gamma)}(1)$ are all within the zero set $\mathcal{Z}\left(f^s_{\scriptscriptstyle{\Omega_1}}\right)$.
\end{proof}
 
Now we introduce a new concept of path-slice analytic  set  within  a slice-domain $\Omega$ in $\mathbb{H}_s^n$.

 \begin{defn}\label{def-path slice discrete}
 	A subset $A$ of a slice-domain $\Omega \subset \mathbb{H}_s^n$ is termed path-slice analytic if it either encompasses the entire domain $\Omega$ or, for each path $\gamma \in \mathscr{P}(\mathbb{C}^n, \Omega)$, certain conditions are met. Specifically, there exists a radius $r > 0$, a set of radii $r{[I]} \in (0, r]$ for each $I \in \mathbb{S}(\Omega, \gamma)$, and an analytic subset $E$ of the ball $B_{\mathbb{C}^n}(\gamma(1), r)$, distinct from the entire ball, such that 
 	\begin{equation*}
 		A \cap B_I\left(\gamma^I(1), r_{[I]}\right) \subseteq E^I
 	\end{equation*}
 for any $  I \in \mathbb{S}(\Omega, \gamma).$	Here, $E^I$ denotes the slice of the set $E$ corresponding to $I \in \mathbb{S}(\Omega, \gamma)$.
 \end{defn}

 We aim to establish that the zeros of a path-slice regular function   constitute a path-slice analytic set. To achieve this, we will need several lemmas.

\begin{lem}\label{prop-zero path-slice discrete}
	Let $\Omega_1\subset\mathbb{H}_s^n$ be a real-path-connected slice-domain, $\Omega_2\in\tau_s(\mathbb{H}_s^n)$ be $\Omega_1$-stem-preserving, and $f\in\mathcal{SR}(\Omega_2)$. Then $\mathcal{Z}\left(f^s_{\scriptscriptstyle{\Omega_1}}\right)$ is path-slice analytic.
\end{lem}

\begin{proof}
	According to Theorem \ref{thm-f c slice regular}, $f^c_{\scriptscriptstyle{\Omega_1}}$ is slice regular. It follows from Theorem \ref{thm-slice regular algebra} that $f^s_{\scriptscriptstyle{\Omega_1}}=f^c_{\scriptscriptstyle{\Omega_1}}*f$ is also slice regular. Let $\gamma\in\mathscr{P}(\mathbb{C}^n,\Omega_1)$, $I\in\mathbb{S}(\Omega_1,\gamma)$, $r\in \left(0,r_{\gamma,\Omega}^{\{I\}}\right)$,
	\begin{equation*}
		r_{[J]}=\min\left\{r,r_{\gamma,\Omega}^{\{J\}}\right\},\qquad\forall\ J\in\mathbb{S}(\Omega_1,\gamma),
	\end{equation*}
	and
	\begin{equation*}
		E_{[J]}:=\mathcal{Z}(f)\cap B_J(\gamma^J(1),r_{[J]}),\qquad\forall\ J\in\mathbb{S}(\Omega_1,\gamma).
	\end{equation*}

	If there is no analytic set $E$ in $B_{\mathbb{C}^n}(\gamma(1),r)$ such that $E\neq B_{\mathbb{C}^n}(\gamma(1),r)$ and $E_{[I]}\subset E^I$. Then by \cite[Splitting Lemma 3.3]{Dou2023002} and Identity Principle in complex analysis, $B_{\mathbb{C}^n}(\gamma(1),r)\subset\mathcal{Z}\left(f^s_{\scriptscriptstyle{\Omega_1}}\right)$. According to \cite[Indentity Principle 3.5]{Dou2023002}, $f^s_{\scriptscriptstyle{\Omega_1}}\equiv 0$. It implies that $\mathcal{Z}\left(f^s_{\scriptscriptstyle{\Omega_1}}\right)=\Omega$ is path-slice analytic.
	
	Otherwise, let $E$ be an analytic set in $B_{\mathbb{C}^n}(\gamma(1),r)$ with $E\neq B_{\mathbb{C}^n}(\gamma(1),r)$ and $E_{[I]}\subset E^I$, and let $z^J\in E_{[J]}$. Then
	\begin{equation*}
		\left(\gamma\circ\mathcal{L}_{\gamma(1)}^z\right)^J(1)=z^J\in\mathcal{Z}\left(f^s_{\scriptscriptstyle{\Omega_1}}\right).
	\end{equation*}
	By Lemma \ref{pr-zero circle},
	\begin{equation*}
		z^I=\left(\gamma\circ\mathcal{L}_{\gamma(1)}^z\right)^I(1)\in\left(\gamma\circ\mathcal{L}_{\gamma(1)}^z\right)^{\mathbb{S}(\Omega,\gamma)}(1)\subset\mathcal{Z}\left(f^s_{\scriptscriptstyle{\Omega_1}}\right).
	\end{equation*}
	It implies that $z^I\in E_{[I]}$, $z\in \left(\Psi_i^I\right)^{-1}\left(E_{[I]}\right)\subset E$ and $z^J\in E^J$. Therefore, $E_{[J]}\subset E^J$ and
	\begin{equation*}
		\mathcal{Z}\left(f^s_{\scriptscriptstyle{\Omega_1}}\right)\cap B_I(\gamma^I(1),r)=E_{[J]}\subset E^J.
	\end{equation*}
	By Definition \ref{def-path slice discrete}, $\mathcal{Z}\left(f^s_{\scriptscriptstyle{\Omega_1}}\right)$ is path-slice analytic.
\end{proof}

\begin{prop}\label{prop-f c}
	Let $\Omega_1\subset\mathbb{H}_s^n$ be a real-path-connected slice-domain, $\Omega_2\in\tau_s(\mathbb{H}_s^n)$ be $\Omega_1$-stem-preserving, and $f\in\mathcal{SR}(\Omega_2)$. Then
		\begin{equation}\label{eq-f c f s on R}
		f^c_{\scriptscriptstyle{\Omega_1}}=\overline{f}\qquad\mbox{and}\qquad f^s_{\scriptscriptstyle{\Omega_1}}=|f|^2\qquad\mbox{on}\qquad \Omega_{\mathbb{R}}.
	\end{equation}
\end{prop}

\begin{proof}
	According to \eqref{eq-lfo1} and \eqref{eq-f Omega1 c def},
	\begin{equation*}
		\begin{split}
			f^c_{\scriptscriptstyle{\Omega_1}}
			=&(1,\mathfrak{I}) \mathscr{F}_{\Omega_1}^{f,c}
			=(1,\mathfrak{I}){\conj}_{\mathbb{H}}\circ\mathscr{F}_{\Omega_1}^{f}
			\\=&(1,\mathfrak{I}){\conj}_{\mathbb{H}}\circ\begin{pmatrix}
				f\\0
			\end{pmatrix}
			=(1,\mathfrak{I})\begin{pmatrix}
				\overline{f}\\0
			\end{pmatrix}
			=\overline{f},
		\end{split}\qquad\mbox{on}\qquad\Omega_{\mathbb{R}}.
	\end{equation*}
	By \eqref{eq-starproduct def} and \eqref{eq-lfo1},
	\begin{equation*}
		f^s_{\scriptscriptstyle{\Omega_1}}
		=f^c_{\scriptscriptstyle{\Omega_1}}*f
		=(f^c_{\scriptscriptstyle{\Omega_1}},\mathfrak{I}f^c_{\scriptscriptstyle{\Omega_1}}) \mathscr{F}_{\Omega_1}^{f}
		=(\overline{f},\mathfrak{I}\overline{f})\begin{pmatrix} f\\0 \end{pmatrix}
		=\overline{f}f=|f|^2,\qquad\mbox{on}\qquad\Omega_{\mathbb{R}}.
	\end{equation*}
	
\end{proof}

Ultimately, we are able to demonstrate that the zeros of a path-slice regular function  form a path-slice analytic set.

\begin{thm}\label{thm-analytic}
	Let $\Omega\subset\mathbb{H}_s^n$ be a self-stem-preserving slice-domain, and $f\in\mathcal{SR}(\Omega)$. Then $\mathcal{Z}(f)$ is path-slice analytic.
\end{thm}

\begin{proof}
	If $\mathcal{Z}\left(f^s_{\scriptscriptstyle{\Omega}}\right)=\Omega$, then by \eqref{eq-f c f s on R},
	\begin{equation*}
		0\equiv f^s_{\scriptscriptstyle{\Omega}}=|f|^2,\qquad\mbox{on}\qquad\Omega_{\mathbb{R}}.
	\end{equation*}
	It implies that $f\equiv 0$ on $\Omega_{\mathbb{R}}$. According to \cite[Indentity Principle 3.5]{Dou2023002},  $f\equiv 0$. It implies that $\mathcal{Z}\left(f\right)=\Omega$ is path-slice analytic.
	
	Otherwise, $\mathcal{Z}\left(f^s_{\scriptscriptstyle{\Omega}}\right)\neq\Omega$. According to \eqref{eq-zero}, $\mathcal{Z}(f)\subset\mathcal{Z}\left(f^s_{\scriptscriptstyle{\Omega}}\right)$. It follows from Proposition \ref{prop-zero path-slice discrete} that $\mathcal{Z}\left(f^s_{\scriptscriptstyle{\Omega}}\right)$ is path-slice analytic with $\mathcal{Z}\left(f^s_{\scriptscriptstyle{\Omega}}\right)\neq\Omega$, so is $\mathcal{Z}(f)\subset\mathcal{Z}\left(f^s_{\scriptscriptstyle{\Omega}}\right)$.
\end{proof}

\bibliographystyle{plain}
\bibliography{mybibfile}

\begin{bibdiv}
\begin{biblist}

\bib{Alpay2015001}{article}{
      author={Alpay, Daniel},
      author={Colombo, Fabrizio},
      author={Gantner, Jonathan},
      author={Sabadini, Irene},
       title={A new resolvent equation for the {$S$}-functional calculus},
        date={2015},
        ISSN={1050-6926,1559-002X},
     journal={J. Geom. Anal.},
      volume={25},
      number={3},
       pages={1939\ndash 1968},
         url={https://doi.org/10.1007/s12220-014-9499-9},
      review={\MR{3358079}},
}

\bib{Alpay2012001}{article}{
      author={Alpay, Daniel},
      author={Colombo, Fabrizio},
      author={Sabadini, Irene},
       title={Schur functions and their realizations in the slice
  hyperholomorphic setting},
        date={2012},
        ISSN={0378-620X},
     journal={Integral Equations Operator Theory},
      volume={72},
      number={2},
       pages={253\ndash 289},
         url={https://doi.org/10.1007/s00020-011-1935-7},
      review={\MR{2872478}},
}

\bib{MR3967697}{book}{
      author={Colombo, Fabrizio},
      author={Gantner, Jonathan},
       title={Quaternionic closed operators, fractional powers and fractional
  diffusion processes},
      series={Operator Theory: Advances and Applications},
   publisher={Birkh\"{a}user/Springer, Cham},
        date={2019},
      volume={274},
        ISBN={978-3-030-16408-9; 978-3-030-16409-6},
         url={https://doi.org/10.1007/978-3-030-16409-6},
      review={\MR{3967697}},
}

\bib{MR3887616}{book}{
      author={Colombo, Fabrizio},
      author={Gantner, Jonathan},
      author={Kimsey, David~P.},
       title={Spectral theory on the {S}-spectrum for quaternionic operators},
      series={Operator Theory: Advances and Applications},
   publisher={Birkh\"{a}user/Springer, Cham},
        date={2018},
      volume={270},
        ISBN={978-3-030-03073-5; 978-3-030-03074-2},
      review={\MR{3887616}},
}

\bib{MR4496722}{article}{
      author={Colombo, Fabrizio},
      author={Gantner, Jonathan},
      author={Kimsey, David~P.},
      author={Sabadini, Irene},
       title={Universality property of the {$S$}-functional calculus,
  noncommuting matrix variables and {C}lifford operators},
        date={2022},
        ISSN={0001-8708,1090-2082},
     journal={Adv. Math.},
      volume={410},
       pages={Paper No. 108719, 39},
         url={https://doi.org/10.1016/j.aim.2022.108719},
      review={\MR{4496722}},
}

\bib{Colombo2009001}{article}{
      author={Colombo, Fabrizio},
      author={Gentili, Graziano},
      author={Sabadini, Irene},
      author={Struppa, Daniele},
       title={Extension results for slice regular functions of a quaternionic
  variable},
        date={2009},
        ISSN={0001-8708},
     journal={Adv. Math.},
      volume={222},
      number={5},
       pages={1793\ndash 1808},
         url={https://doi.org/10.1016/j.aim.2009.06.015},
      review={\MR{2555912}},
}

\bib{Colombo2009002}{article}{
      author={Colombo, Fabrizio},
      author={Sabadini, Irene},
      author={Struppa, Daniele~C.},
       title={Slice monogenic functions},
        date={2009},
        ISSN={0021-2172},
     journal={Israel J. Math.},
      volume={171},
       pages={385\ndash 403},
         url={https://doi.org/10.1007/s11856-009-0055-4},
      review={\MR{2520116}},
}

\bib{Colombo2012002}{article}{
      author={Colombo, Fabrizio},
      author={Sabadini, Irene},
      author={Struppa, Daniele~C.},
       title={Algebraic properties of the module of slice regular functions in
  several quaternionic variables},
        date={2012},
        ISSN={0022-2518},
     journal={Indiana Univ. Math. J.},
      volume={61},
      number={4},
       pages={1581\ndash 1602},
         url={https://doi.org/10.1512/iumj.2012.61.4978},
      review={\MR{3085619}},
}

\bib{Dickson1919001}{article}{
      author={Dickson, Leonard~E.},
       title={On quaternions and their generalization and the history of the
  eight square theorem},
        date={1919},
        ISSN={0003-486X},
     journal={Ann. of Math. (2)},
      volume={20},
      number={3},
       pages={155\ndash 171},
         url={https://doi.org/10.2307/1967865},
      review={\MR{1502549}},
}

\bib{Dou2023004}{article}{
      author={Dou, Xinyuan},
      author={Jin, Ming},
      author={Ren, Guangbin},
      author={Yang, Ting},
       title={Algebra of slice regular functions on non-symmetric domains in
  several quaternionic variables},
        date={2024},
      eprint={arXiv:2401.04895},
}

\bib{Dou2023003}{article}{
      author={Dou, Xinyuan},
      author={Jin, Ming},
      author={Ren, Guangbin},
      author={Yang, Ting},
       title={Path-slice star-product on non-axially symmetric domains in
  several quaternionic variables},
        date={2024},
      eprint={arXiv:2401.04401},
}

\bib{Dou2023001}{article}{
      author={Dou, Xinyuan},
      author={Ren, Guangbin},
      author={Sabadini, Irene},
       title={Extension theorem and representation formula in
  non-axially-symmetric domains for slice regular functions},
        date={2023},
        ISSN={1435-9855,1435-9863},
     journal={J. Eur. Math. Soc. (JEMS)},
      volume={25},
      number={9},
       pages={3665\ndash 3694},
         url={https://doi.org/10.4171/jems/1260},
      review={\MR{4634680}},
}

\bib{Dou2023002}{article}{
      author={Dou, Xinyuan},
      author={Ren, Guangbin},
      author={Sabadini, Irene},
       title={A representation formula for slice regular functions over
  slice-cones in several variables},
        date={2023},
        ISSN={0373-3114,1618-1891},
     journal={Ann. Mat. Pura Appl. (4)},
      volume={202},
      number={5},
       pages={2421\ndash 2446},
         url={https://doi.org/10.1007/s10231-023-01325-y},
      review={\MR{4634271}},
}

\bib{Fueter1934001}{article}{
      author={Fueter, Rudolf},
       title={Die {F}unktionentheorie der {D}ifferentialgleichungen {$\Delta
  u=0$} und {$\Delta\Delta u=0$} mit vier reellen {V}ariablen},
        date={1934},
        ISSN={0010-2571},
     journal={Comment. Math. Helv.},
      volume={7},
      number={1},
       pages={307\ndash 330},
         url={https://doi.org/10.1007/BF01292723},
      review={\MR{1509515}},
}

\bib{Gantner2020001}{article}{
      author={Gantner, Jonathan},
       title={Operator theory on one-sided quaternion linear spaces: intrinsic
  {$S$}-functional calculus and spectral operators},
        date={2020},
        ISSN={0065-9266},
     journal={Mem. Amer. Math. Soc.},
      volume={267},
      number={1297},
       pages={iii+101},
         url={https://doi.org/10.1090/memo/1297},
      review={\MR{4195727}},
}

\bib{Gentili2008001}{article}{
      author={Gentili, Graziano},
      author={Stoppato, Caterina},
       title={Zeros of regular functions and polynomials of a quaternionic
  variable},
        date={2008},
        ISSN={0026-2285},
     journal={Michigan Math. J.},
      volume={56},
      number={3},
       pages={655\ndash 667},
         url={https://doi.org/10.1307/mmj/1231770366},
      review={\MR{2490652}},
}

\bib{MR3026135}{incollection}{
      author={Gentili, Graziano},
      author={Stoppato, Caterina},
       title={The zero sets of slice regular functions and the open mapping
  theorem},
        date={2011},
   booktitle={Hypercomplex analysis and applications},
      series={Trends Math.},
   publisher={Birkh\"{a}user/Springer Basel AG, Basel},
       pages={95\ndash 107},
         url={https://doi.org/10.1007/978-3-0346-0246-4_7},
      review={\MR{3026135}},
}

\bib{MR4182982}{article}{
      author={Gentili, Graziano},
      author={Stoppato, Caterina},
       title={Geometric function theory over quaternionic slice domains},
        date={2021},
        ISSN={0022-247X,1096-0813},
     journal={J. Math. Anal. Appl.},
      volume={495},
      number={2},
       pages={Paper No. 124780, 38},
         url={https://doi.org/10.1016/j.jmaa.2020.124780},
      review={\MR{4182982}},
}

\bib{Gentili2007001}{article}{
      author={Gentili, Graziano},
      author={Struppa, Daniele~C.},
       title={A new theory of regular functions of a quaternionic variable},
        date={2007},
        ISSN={0001-8708},
     journal={Adv. Math.},
      volume={216},
      number={1},
       pages={279\ndash 301},
         url={https://doi.org/10.1016/j.aim.2007.05.010},
      review={\MR{2353257}},
}

\bib{Gentili2010001}{article}{
      author={Gentili, Graziano},
      author={Struppa, Daniele~C.},
       title={Regular functions on the space of {C}ayley numbers},
        date={2010},
        ISSN={0035-7596},
     journal={Rocky Mountain J. Math.},
      volume={40},
      number={1},
       pages={225\ndash 241},
         url={https://doi.org/10.1216/RMJ-2010-40-1-225},
      review={\MR{2607115}},
}

\bib{Ghiloni2011001}{article}{
      author={Ghiloni, Riccardo},
      author={Perotti, Alessandro},
       title={Slice regular functions on real alternative algebras},
        date={2011},
        ISSN={0001-8708},
     journal={Adv. Math.},
      volume={226},
      number={2},
       pages={1662\ndash 1691},
         url={https://doi.org/10.1016/j.aim.2010.08.015},
      review={\MR{2737796}},
}

\bib{Ghiloni2012001}{article}{
      author={Ghiloni, Riccardo},
      author={Perotti, Alessandro},
       title={Slice regular functions of several {C}lifford variables},
        date={2012},
     journal={AIP Conference Proceedings},
      volume={1493, 734},
         url={https://doi.org/10.1063/1.4765569},
}

\bib{Ghiloni2020001}{article}{
      author={Ghiloni, Riccardo},
      author={Perotti, Alessandro},
       title={Slice regular functions in several variables},
        date={2020},
      eprint={arXiv:2007.14925},
}

\bib{Ren2017001}{article}{
      author={Ren, Guangbin},
      author={Wang, Xieping},
       title={Growth and distortion theorems for slice monogenic functions},
        date={2017},
        ISSN={0030-8730},
     journal={Pacific J. Math.},
      volume={290},
      number={1},
       pages={169\ndash 198},
         url={https://doi.org/10.2140/pjm.2017.290.169},
      review={\MR{3673083}},
}

\bib{Ren2017002}{article}{
      author={Ren, Guangbin},
      author={Wang, Xieping},
       title={Julia theory for slice regular functions},
        date={2017},
        ISSN={0002-9947},
     journal={Trans. Amer. Math. Soc.},
      volume={369},
      number={2},
       pages={861\ndash 885},
         url={https://doi.org/10.1090/tran/6717},
      review={\MR{3572257}},
}

\bib{Wang2017001}{article}{
      author={Wang, Xieping},
      author={Ren, Guangbin},
       title={Boundary {S}chwarz lemma for holomorphic self-mappings of
  strongly pseudoconvex domains},
        date={2017},
        ISSN={1661-8254},
     journal={Complex Anal. Oper. Theory},
      volume={11},
      number={2},
       pages={345\ndash 358},
         url={https://doi.org/10.1007/s11785-016-0552-5},
      review={\MR{3605232}},
}

\end{biblist}
\end{bibdiv}

\end{document}